\documentclass[draft,11pt]{amsart}

\usepackage{amssymb}
\usepackage{amsmath, amsfonts,enumerate}
\usepackage{hyperref}
\usepackage{graphics}
\usepackage{amsthm}
\usepackage{latexsym,bm}
\usepackage{euscript}

\let\cal\mathcal

\makeatletter \@mparswitchfalse \makeatother
\newtheorem{theorem}{Theorem}
\newtheorem{lemma}[theorem]{Lemma}

\newtheorem{proposition}[theorem]{Proposition}
\theoremstyle{remark}
\newtheorem{remark}[theorem]{Remark}

\theoremstyle{definition}
\newtheorem{definition}[theorem]{Definition}
\newtheorem{problem}[theorem]{Problem}

\theoremstyle{remark}

\numberwithin{equation}{section}
\numberwithin{theorem}{section}

\def\M{\cal{M}}
\def\H{\cal{H}}
\let\<\langle
\def\ch{\raise 0.5ex \hbox{$\chi$}}
\def\T{\tau}
\def\E{\cal{E}}

\let\\\cr

\let\epsilon\varepsilon

\def\log{\operatorname{log}}

\renewcommand{\a}{\alpha}
\renewcommand{\b}{\beta}
\newcommand{\g}{\gamma}

\renewcommand{\i}{{\rm i}}

\newcommand{\N}{\cal{N}}
\newcommand{\h}{\mathsf{h}}

\newcommand{\bmo}{\mathsf{bmo}}
\newcommand{\BMO}{{\mathcal {BMO}}}

\begin{document}

\title[Interpolation]{P.~Jones'  interpolation  theorem for noncommutative martingale Hardy spaces II}

\author[Randrianantoanina]{Narcisse Randrianantoanina}
\address{Department of Mathematics, Miami University, Oxford,
Ohio 45056, USA}
 \email{randrin@miamioh.edu}



\begin{abstract}
Let $\M$ be a semifinite von Neumann algebra equipped with an increasing filtration $(\M_n)_{n\geq 1}$  of (semifinite) von Neumann subalgebras  of $\M$.  For
  $1\leq p \leq\infty$,  let  $\H_p^c(\M)$ denote the  noncommutative  column    martingale Hardy space  constructed from column square functions  associated with    the  filtration $(\M_n)_{n\geq 1}$ and the index $p$.
We prove the following  real interpolation identity: if $0<\theta<1$ and $1/p=1-\theta$, then
   \[
   \big(\H_1^c(\M), \H_\infty^c(\M)\big)_{\theta,p}=\H_p^c(\M).
   \]
   This is new even for classical martingale Hardy spaces as it is previously known only under the assumption that the filtration is regular.
We also obtain   analogous result  for  noncommutative  column martingale Orlicz-Hardy spaces.
\end{abstract}

\maketitle

\section{Introduction}

This article  is a continuation of our previous work \cite{Ran-Int2} on  the investigation
of  interpolation spaces of compatible couples of various types of noncommutative martingale Hardy spaces and related spaces. The  study of interpolations of classical martingale Hardy spaces  has  a long history   which we refer to \cite{Jason-Jones,Weisz3, Weisz, Xu-inter2}. Recall that  the first  class  of noncommutative  martingale Hardy spaces were introduced by  Pisier and Xu in the seminal  paper \cite{PX}.  These  Hardy spaces were constructed from  column/row  square functions  and  the column versions are  commonly denoted by $\H_p^c$ for $0<p\leq\infty$. Later, another class of noncommutative martingale Hardy spaces
were considered  by Junge and Xu  in \cite{JX} which are based on conditioned column/row square functions. The column versions of these second  type of    Hardy space  are  denoted by $\h_p^c$  for $0<p\leq \infty$. Both classes of Hardy spaces are instrumental  in  the development of noncommutative martingale theory the last two decades.
The topic of interpolations of these classes of Hardy spaces  turns out to be at the forefront of these developments.

Let us recall  some background  on interpolations of noncommutative martingale Hardy spaces. We refer to \cite{JX,PX} for the classes of noncommutative  BMO-spaces discussed below. 
The study of interpolations of noncommutative martingale Hardy spaces was  initiated by Musat in \cite{Musat-inter} 
where the complex interpolation of the compatible  couple $(\H_1^c, \BMO^c)$ was given. Later, Bekjan  {\it et al.} established in \cite{Bekjan-Chen-Perrin-Y} that  the analogue of  Musat's result is valid for  couple of  column/row conditioned Hardy spaces. More precisely,  they   obtained  the corresponding result for the compatible couple $(\h_1^c, \bmo^c)$. A common theme in  these two earlier  articles   is that  all couples considered have one of the endpoints consisting of  appropriate types of noncommutative  martingale $BMO$-spaces. We should also mention the articles \cite{Bekjan-Int, Bekjan-Chen-R-S} for related interpolation results in this general direction.

 In the recent article \cite{Ran-Int2}, we 
obtained a Peter Jones  type result for real interpolations of the  couple $(\h_p^c, \h_\infty^c)$ for  $0<p <\infty$. We refer to  \cite{BENSHA, Jones1, Jones2, Pisier-int} for background concerning Peter Jones's interpolation result. It is a natural question if 
the results from \cite{Ran-Int2}  remains valid for other classes of Hardy spaces.

\smallskip

The primary objective of the present article  is to investigate the compatible couple  $(\H_1^c,  \H_\infty^c)$. Curiously,  general interpolation  results  for this  type of couples of  Hardy spaces  from classical martingale theory do not appear to be available in the literature. In fact, all  existing results  to date in this direction  require that  either the filtration  of $\sigma$-algebras involved  is regular or  the  Hardy spaces considered  were  of special type such as those restricted to martingales whose sequences of  quadratic variations are
dominated by predictable sequence of random variables.  Our findings  remedy this situation for both commutative and noncommutative settings. More precisely, we establish   a general result which states  that as in the conditioned case, the family  of Hardy spaces $\{\H_p^c\}_{1\leq p\leq \infty}$ forms a real interpolation scale. In other words, we obtain that   if $0<\theta<1$ and $1/p=1-\theta$, then (with equivalent norms)
   \begin{equation}\label{big-h}
   \big(\H_1^c, \H_\infty^c\big)_{\theta,p}=\H_p^c
   \end{equation}
   where $(\cdot, \cdot)_{\theta,q}$ denotes the real interpolation method. 
We should point out  that,  as is already known from the classical case, the identity \eqref{big-h} does not extend to the compatible couple  $(\H_p^c, \H_\infty^c)$ for $0<p<1$. This is in strong contrast with the conditioned case from \cite{Ran-Int2} and thus highlighting  
that the two situations can be quite different. However, our method  of proof is  based on insights  from techniques used in \cite{Ran-Int2}. Indeed, the decisive step in  our  argument  is an 
 estimate from above of the $K$-functional of  the couple $(\H_2^c, \H_\infty^c)$.
We refer  to Theorem~\ref{main} below for details. The main feature of the estimate is  our  use of the dual Ces\`aro operator.   Although this particular estimate does not imply that  the couple $(\H_2^c, \H_\infty^c)$ is $K$-closed in some appropriate  noncommutative couple $(L_2, L_\infty)$ as in the case of  conditioned spaces treated  in  \cite{Ran-Int2},  it is sufficient to deduce  real interpolations of the couple  $(\H_2^c, \H_\infty^c)$.  The  case of the couple $(\H_1^c, \H_\infty^c)$ is then deduced from the couple $(\H_2^c,\H_\infty^c)$ via standard use of  Wolff's interpolation theorem.

\smallskip

Motivated by  the recent  interest on   martingale Hardy spaces associated with Orlicz  function spaces in the classical setting, we  also consider the case of noncomutative Orlicz-Hardy spaces.
We refer to  recent articles  \cite{L-T-Zhou, L-Weisz-Xie} for  more perspective   and background on real interpolations of  compatible couples involving  martingale Orlicz-Hardy spaces in the classical setting. We show as an  extension of our  techniques 
that \eqref{big-h} remains valid   for noncommutative martingale  Hardy spaces associated with Orlicz functions satisfying some natural conditions.

\smallskip

The paper is organized as follows. In Section~2, we give a brief introduction  of noncommutative  spaces and  review the construction of  noncommutative martingale Hardy spaces associated with symmetric  spaces of measurable operators.  This  section also contains  relevant discussions on  some concepts from interpolation  theory. More specifically,  the real interpolation method is discussed. It also includes some background concerning  the Ces\`aro operator and its dual operator that play  key parts  in  the estimate of the $K$-functional discussed above.
Section~3 is where we provide the formulation  and proof  of our primary  result  together with  extensions and applications. The section also contains a  paragraph dealing with noncommutative martingale Hardy spaces  associated with general symmetric spaces. More specifically,  we consider reeal interpolations of couples involving  classes  of noncommutative matingale  Hardy spaces associated with Orlicz function spaces. 


\section{definitions and preliminary results}

Throughout, we use $c_{abs}$ to denote some absolute constant whose value may change from one statement to the next.  We write $A\lesssim B$ if there is some absolute constant $c_{abs}$ such that $A\leq c_{abs} B$.
We say that $A$ is equivalent to $B$  if $A\lesssim B$ and $B\lesssim A$. In this case, we write $A\approx B$. 

\subsection{Generalized singular value functions and noncommutative spaces}\label{svf subsection}

In what follows, $H$ is a separable  Hilbert space and $\M\subseteq \cal{B(}H)$ denotes a semifinite  von Neumann algebra equipped with a faithful normal semifinite trace $\T$. The identity in $\M$ will be denoted by ${\bf 1}$. A closed and densely defined operator $a$ on $H$ is said to be {\it affiliated} with $\mathcal{M}$ if $u^{\ast}au=a$ for each unitary operator $u$ in the commutant $\mathcal{M}'$ of $\mathcal{M}$. An operator $x$ is called {\it $\T$-measurable} if $x$ is affiliated with $\mathcal{M}$ and for every $\varepsilon>0$, there exists a projection $p\in \mathcal{M}$ such that $p(H)\subset {\mathrm{dom}}(x)$ and $\tau({\bf 1}-p)<\varepsilon.$
The set of all  $\tau$-measurable operators will be denoted by $L_0(\mathcal{M})$. Given a self-adjoint operator $x\in L_0(\M)$ and  a Borel set $B\subset\mathbb{R}$, we denote by $\ch_{B}(x)$ its spectral projection. 
  The {\it distribution function} of $x$ is defined by
\[
\lambda_s(x)=\tau\left(\chi_{(s,\infty)}(x)\right), \quad  s \in \mathbb{R}.
\]
For $x\in L_0(\M)$, the {\it generalized singular value function} of $x$ is defined by
\[
\mu_t(x)=\inf\left\{s>0: \lambda_s(|x|)\leq t\right\}, \quad t>0.
\]

The function $t\mapsto\mu_t(x)$ is decreasing and right-continuous. In the case that $\M$ is the abelian von Neumann algebra $L_{\infty}(0,\alpha)$ ($0<\alpha\leq \infty$) with the trace given by  the integration with respect to the Lebesgue measure,  $L_0(\mathcal{M})$ is the space of all measurable functions, with non-trivial distribution, and $\mu(f)$ is the decreasing rearrangement of the  measurable function $f$ (see  \cite{KPS}). In the abelian case, we write $L_0(0,\alpha)$ instead of
$L_0(L_\infty(0,\alpha))$ ($0<\alpha\leq \infty$). For more discussions on generalized singular value functions, we refer the reader to
\cite{FK}.

Let $0<\alpha\leq \infty$. A  (quasi) Banach  function space $(E,\|\cdot\|_E)$ on the interval $(0,\alpha)$  is called {\it symmetric} if for every $g\in E$ and for every measurable function $f$ with $\mu(f)\leq\mu(g)$, we have $f\in E$ and $\|f\|_E\leq\|g\|_E$.

Given  a symmetric quasi-Banach function space $E$ on $(0,\alpha)$, we  define the  corresponding  noncommutative space of operators  by setting:
\begin{equation*}
E(\M, \T) = \Big\{ x \in
L_0(\M)\ : \ \mu(x) \in E \Big\}. 
\end{equation*}
Equipped with the  quasi-norm
$\|x\|_{E(\M,\T)} := \| \mu(x)\|_E$,   the linear space $E(\M,\T)$ becomes a complex quasi-Banach space (\cite{Kalton-Sukochev,Sukochev-quasi,X}) and is usually referred to as the \emph{noncommutative symmetric space} associated with $(\M,\T)$ corresponding to  $(E, \|\cdot\|_E)$. 
 We remark that  if $0< p<\infty$ and $E=L_p$, then $E(\M, \T)$ is exactly   the usual noncommutative $L_p$-space  $L_p(\M,\T)$ associated with  the pair $(\M,\T)$. 
 In the sequel, $E(\M,\T)$ will be abbreviated to $E(\M)$. 
 
 Beside $L_p$-spaces, we also  make extensive use of  two classes of symmetric spaces. Namely, Lorentz spaces and Orlicz spaces. We begin with the former.

$\bullet$\ {\it Lorentz spaces}:
Let $0<p,q\leq \infty$. The \emph{Lorentz space} $L_{p,q}(0,\infty)$ is the space  of all $f \in L_0(0,\infty)$ for which $\|f\|_{p,q}<\infty$ where
 \begin{equation*}
 \big\|f \big\|_{p,q}  =\begin{cases}
 \left(\displaystyle{\int_{0}^\infty \mu_{t}^{q}(f)\
d(t^{q/p})}\right)^{1/q},  &0< q < \infty; \\
\displaystyle{\sup_{t >0} t^{1/p} \mu_t(f)}, &q= \infty.
\end{cases} 
 \end{equation*}
If $1\leq q\leq p <\infty$ or $p=q=\infty$, then $L_{p,q}(0,\infty)$ is a symmetric Banach function space.  If $1<p<\infty$ and $p\leq q\leq \infty$, then $L_{p,q}(0,\infty)$ can be equivalently   renormed to become a symmetric Banach function space (\cite[Theorem~4.6]{BENSHA}). In general, $L_{p,q}(0,\infty)$ is only a symmetric quasi-Banach function space.

We will also use a different type of Lorentz spaces which we briefly describe. Let $\phi:[0,\infty)\to [0,\infty)$ be an increasing concave continuous function such that $\phi(0)=0$ and $\lim_{t\to \infty}\phi(t)=\infty$.
The Lorentz  space $\Lambda_{\phi}(0,\infty)$ is defined by setting:
\[
\Lambda_{\phi}(0,\infty):=\left\{f\in L_0(0,\infty):\int_0^\infty\mu_s(f)\, d\phi(s)<\infty\right\}
\]
equipped with the norm
\[\|f\|_{\Lambda_{\phi}(0,\infty)}:=\int_0^\infty\mu_s(f )\, d\phi(s).
\]
Clearly, $\Lambda_\phi(0,\infty)$ is a symmetric Banach function space.  In the case $\phi(t)=\log(1+t)$, we use  the notation $\Lambda_{\log}(0,\infty)$ for  $\Lambda_{\phi}(0,\infty)$.

\medskip

$\bullet$\ {\it Orlicz spaces}: By an {\it Orlicz function} $\Phi$ on $[0,\infty)$, we mean a continuous increasing  function satisfying $\Phi(0)=0$ and $\lim_{t\to \infty} \Phi(t)=\infty.$ An Orlicz function $\Phi$ is said to be {\it $p$-convex} if the function $t \mapsto\Phi(t^{1/p})$ is convex, and to be {\it$q$-concave} if the function $t\mapsto \Phi(t^{1/q})$ is concave. For a given Orlicz function $\Phi$  that is $p$-convex and $q$-concave for $0<p\leq q <\infty$, the associated {\it Orlicz space} $L_{\Phi}(0,\infty)$ is defined by setting
$$L_{\Phi}(0,\infty):=\left\{f\in L_0(0,\infty):\ \int_0^\infty \Phi\Big(\frac{|f(s)|}{\lambda}\Big)\, ds<\infty\mbox{ for some }\lambda>0\right\}$$
equipped with the quasi-norm
$$\|f\|_\Phi:=\inf\left\{\lambda>0:\int_0^\infty \Phi\Big(\frac{|f(s)|}{\lambda}\Big)\, ds\leq 1\right\}.$$

The Orlicz space  $L_\Phi(0,\infty)$ is a symmetric  quasi-Banach function space.  We refer to \cite{Kras-Rutickii, Maligranda2} for more details on  Orlicz functions and  Orlicz spaces.
We will also make use of the  following more general space:  for $0<r\leq \infty$, the space $L_{\Phi,r}(0,\infty)$ is the collection of all $f \in L_0(0,\infty)$ for which 
$\| f\|_{\Phi,r} <\infty$ where 
\begin{equation*}
\big\|f \big\|_{\Phi,r} :=\begin{cases}
 \left(r\displaystyle{\int_{0}^\infty \big(t \| \ch_{\{|f|>t\}} \|_\Phi}\big)^r\ \frac{dt}{t}\right)^{1/r},  &0< r < \infty; \\
\displaystyle{\sup_{t >0} t \| \ch_{\{|f|>t\}} \|_\Phi}, &r= \infty.
\end{cases} 
 \end{equation*}
The space $L_{\Phi,r}(0,\infty)$ was introduced in \cite{Hao-Li} and was called Orlicz-Lorentz
space there. Note that if $\Phi(t)=t^p$,  then  $L_{\Phi,r}(0,\infty)$ coincides with  the Lorentz space $L_{p,r}(0,\infty)$. The space $L_{\Phi,\infty}(0,\infty)$ is also known as the weak Orlicz space. 


\medskip

We conclude this subsection  by recalling the notion of \emph{submajorization}. Given $x, y\in L_0(\mathcal M)$, we say that $y$ is {\it submajorized} in the sense of Hardy-Littlewood-P\'{o}lya by $x$ (written $y\prec\prec x$) if
\[
\int_0^t \mu_s(y)\, ds\leq \int_0^t \mu_s(x)\, ds,\quad t>0.
\]

In the sequel, we will frequently use the submajorization  inequality
\begin{equation}\label{sub-sum}
\mu(x+y) \prec\prec \mu(x) +\mu(y), \quad x,y \in L_0(\M).
\end{equation}
Another fact that is important below is that if $T: L_1(\M) +\M \longrightarrow  L_1(\M) +\M$ satisfies 
$\max\{ \| T: L_1(\M) \to L_1(\M)\|; \| T: \M \to \M\|\} \leq 1$ then  for every $x \in L_1(\M) +\M$, $Tx \prec\prec x$. This fact can be found in \cite[Proposition~4.1]{DDP3}. In particular, if $x \in L_1(\M)+\M$ and $(p_k)_{k\geq 1}$  is a sequence of  mutually disjoint projections from $\M$  then, 
\begin{equation}\label{sub-diagonal}
\sum_{k\geq 1} p_k x p_k \prec\prec x.  
\end{equation}



\subsection{Noncommutative martingale Hardy spaces}\label{martingale}
By a filtration  $(\M_n)_{n \geq 1}$,   we mean an
increasing sequence of von Neumann subalgebras of ${\M}$
whose union  is w*-dense in
$\M$.  Throughout, we will work with a fixed filtration $(\M_n)_{n\geq 1}$. For  every $n\geq 1$,  we assume  further that there is a trace preserving conditional expectation $\E_n$
from ${\M}$ onto  ${\M}_n$. This is the case  if  for every $n\geq 1$, the restriction of  the trace $\T$  on $\M_n$ is semifinite. It is well-know that for $1\leq p<\infty$,  the $\E_n$'s  extend to be contractive  projections from $L_p(\M,\T)$ onto $L_p(\M_n,\T|_{\M_n})$. In particular, they are well-defined  on $L_1(\M) +\M$. 

\begin{definition}
A sequence $x = (x_n)_{n\geq 1}$ in $L_1(\M)+\M$ is called \emph{a
noncommutative martingale} with respect to the filtration  $({\M}_n)_{n \geq
1}$ if  for every $n \geq 1$,
\[
\E_n (x_{n+1}) = x_n.
\]
\end{definition}

Let $E$ be a symmetric quasi-Banach function space  and $x=(x_n)_{n\geq 1}$ be a martingale. If  for every $n\geq 1$, $x_n \in E(\M_n)$, then we  say that  $(x_n)_{n\geq 1}$ is an $E(\M)$-martingale.  In this case, we set
\begin{equation*}\| x \|_{E(\M)}= \sup_{n \geq 1} \|
x_n \|_{E(\M)}.
\end{equation*}
If $\| x \|_{E(\M)} < \infty$, then $x$   will be called
a bounded $E(\M)$-martingale.

For a martingale $x=(x_n)_{n\geq 1}$, we set $dx_n=x_n-x_{n-1}$  for $n\geq 1$ with the usual convention that $x_0=0$. The sequence $dx=(dx_n)_{n\geq 1}$ is called the \emph{martingale difference sequence} of $x$. A martingale $x$  is called a \emph{finite martingale} if there exists $N\in \mathbb{N}$ such that $dx_n=0$ for all $n\geq N$.

Let us now  review some basic  definitions related to   martingale Hardy  spaces associated to   noncommutative   symmetric spaces. 

Following \cite{PX}, we define  the  \emph{column square functions} of a given
  martingale $x = (x_k)_{k\geq 1}$ by setting:
 \[
 S_{c,n} (x) = \Big ( \sum^n_{k = 1} |dx_k |^2 \Big )^{1/2}, \quad
 S_c (x) = \Big ( \sum^{\infty}_{k = 1} |dx_k |^2 \Big )^{1/2}\,.
 \]
  For convenience, we will use the notation 
  \[
  \cal{S}_{c,n}(a)= \Big ( \sum^n_{k = 1} |a_k |^2 \Big )^{1/2}, \quad
 \cal{S}_c (a) = \Big ( \sum^{\infty}_{k = 1} |a_k |^2 \Big )^{1/2}
 \]
 for sequences $a=(a_k)_{k\geq 1}$ in $L_1(\M)+\M$     that are not necessarily  martingale difference sequences.  
 It is worth pointing out that the infinite sums of positive operators stated above may not always make sense as operators. However, if the sequence   $(S_{c,n}(x))_{n\geq 1}$  is order bounded,  then it admits a  supremum. In that case, $S_c(x)$ may be taken to be  the limit  of the sequence  $(S_{c,n}(x) )_{n\geq 1}$ for the measure topology.  Similar remark applies to  the sequence $(\cal{S}_{c,n}(a))_{n\geq 1}$.                                                                                                                                                                                                                                                                                                                                                                                                                                                                                                                                                                                                                                                                                                                                                                                                                                                                                                                                                                                                                                                                                                                                                                                                                                                                                                                                                                                                                                                                                                                                                                                                                                                                                                                                                                                                                                                                                                                                                                                                                                                                                                                                                                                                                                                                                                                                                                                                                                                                                                                                                                                                                                                                                                                                                                                                                                                                                                                                                                                                                                                                                                                                                                                                                                                                                                                                                                                                                                                                                                                                                                                                                                                                                                                                                                                                                                                                                                                                                                                                                                                                                             
 
 We will now  describe   noncommutative martingale Hardy spaces associated with symmetric Banach  function spaces.
 In this paper, we will only work with   Hardy spaces built from  square functions. We refer to \cite{Ran-Int,Ran-Int2} for other types  of noncommutative  martingale Hardy spaces.
 
Assume  that $E$ is a symmetric  Banach function space on $(0,\infty)$. We denote by $\cal{F}_E$ the collection  of all finite martingales in $E(\M) \cap \M$.  
For  $x=(x_k)_{k\geq 1} \in \cal{F}_E$, we set:
\[
\big\| x  \big\|_{\mathcal{H}_E^c}= \big\| {S}_c(x) \big\|_{E(\M)} .\]
Then $(\cal{F}_E, \|\cdot\|_{\cal{H}_E^c})$ is a normed space.
 If we denote by $(e_{i,j})_{i,j \geq 1}$   the family of unit matrices in $\cal{B}(\ell_2)$, then 
 the  correspondence  $x\mapsto \sum_{k\geq 1} dx_k \otimes e_{k,1}$ maps $\cal{F}_E$ isometrically into 
a (not necessarily closed) linear subspace of $E(\M\overline{\otimes} \cal{B}(\ell_2))$.

We define the \emph{ column martingale  Hardy space}
  $\mathcal{H}_E^c (\mathcal{M})$ to  be the completion  of $(\cal{F}_E, \|\cdot\|_{\cal{H}_E^c})$.   It then follows that 
$\H_E^c(\M)$  embeds isometrically into a closed subspace of  the quasi-Banach space $E(\M\overline{\otimes} \cal{B}(\ell_2))$. 

We remark that   using the above definition with  $L_p(0,\infty)$ where $1\leq p<\infty$, we recover the definition of $\H_p^c(\M)$ as defined in \cite{PX}. However, the case $p=\infty$ is not covered by the above description since it requires separability.  We define  $\H_\infty^c(\M)$ as the collection of  all martingales in $\M$ for which the column square functions exists in $\M$. The  norm in $\H_\infty^c(\M)$ is defined by:
\[
\big\|x\big\|_{\H_\infty^c}= \big\|S_c(x)\big\|_\infty, \quad x \in \H_\infty^c(\M).
\] 

 In the sequel, we will also make use of the  more general  column space $E(\M;\ell_2^c)$ which is defined as the set of all sequences  $a=(a_k)_{k\geq 1}$ in $E(\M)$ for which $\cal{S}_c(a)$ exists in $E(\M)$.  In this case, we set
\[
\big\|a \big\|_{E(\M;\ell_2^c)} = \|\cal{S}_c(a)  \|_{E(\M)}.
\]
 Under the above quasi-norm, one can easily see that  $E(\M;\ell_2^c)$ is a quasi-Banach space. 
  The closed  subspace of $E(\M;\ell_2^c)$  consisting of adapted sequences will be denoted by $E^{\rm ad}(\M;\ell_2^c)$. That is,
\[
E^{\rm ad}(\M;\ell_2^c)=\Big\{  (a_n)_{n\geq 1} \in E(\M;\ell_2^c) :  \forall n\geq 1, a_n \in E(\M_n) \Big\}.
\]

Below, we use  the notation  $\H_{p,q}^c(\M)$ for the noncommutative column  martingale  Hardy space  associated with the Lorentz space $L_{p,q}(0,\infty)$. Similarly, $\H_{\Phi}^c(\M)$ and $\H_{\Phi,r}^c(\M)$  are used for  noncommutative column  Hardy spaces associated with  the function spaces $L_\Phi(0,\infty)$ and $L_{\Phi,r}(0,\infty)$ respectively.

Note that for  $1<p<\infty$,  it follows from the noncommutative  Stein inequality that $L_p^{\rm ad}(\M;\ell_2^c)$ is a complemented subspace of  $L_p(\M;\ell_2^c)$.  Similarly, for $1\leq p<\infty$, $\H_p^c(\M)$ is a complemented subspace of $L_p^{\rm ad}(\M;\ell_2^c)$. The case $p=1$ is a consequence of the noncommutative L\'epingle-Yor inequality (\cite{Qiu1}). However, in general,  $\H_1^c(\M)$ is not a complemented subspace of $L_1(\M;\ell_2^c)$. Likewise, $\H_\infty^c(\M)$ is not  a complemented  subspace of  $L_\infty(\M;\ell_2^c)$.

\subsection{Basics of interpolations}

Let $(A_0, A_1)$ be a  compatible couple of  quasi-Banach spaces  in the sense that both $A_0$ and $A_1$ embed continuously into some topological vector space $\mathcal{Z}$. This allows us to  define the spaces $A_0 \cap A_1$  and $A_0 +A_1$. These are quasi-Banach spaces when equipped with  quasi-norms:
\[
\big\|x  \big\|_{A_0 \cap A_1}=\max\Big\{ \big\|x  \big\|_{A_0 } , \big\|x  \big\|_{ A_1}\Big\}
\]
and
\[
\big\|x  \big\|_{A_0 + A_1}=\inf\Big\{ \big\|x_0  \big\|_{A_0 } + \big\|x_1  \big\|_{ A_1}: \, x=x_0 +x_1,\,  x_0 \in A_0,\,  x_1 \in A_1\Big\},
\]
respectively.  
\begin{definition}\label{def-interpolation}
A  (quasi) Banach space $A$ is called  an \emph{interpolation space} for the couple $(A_0, A_1)$  if $A_0 \cap A_1 \subseteq A \subseteq A_0 +A_1$ and whenever  a bounded linear operator $T: A_0 +A_1\to A_0 +A_1$ is such that $T(A_0) \subseteq A_0$ and $T(A_1) \subseteq A_1$, we have $T(A)\subseteq A$
and 
\[
\big\|T:A\to A \big\|\leq c\max\left\{\big\|T :A_0 \to A_0\big\| , \big\|T: A_1 \to A_1\big\| \right\}
\] 
for some constant $c$. 
\end{definition}
If $A$ is an interpolation space for the couple $(A_0, A_1)$, we write $A\in {\rm Int}(A_0,A_1)$. Below, we are primarily  interested in an  interpolation method generally referred to as   \emph{the real  interpolation method} which we now  briefly review.
 
 A fundamental notion  for the construction of   real interpolation spaces is the  \emph{$K$-functional}. For $x \in A_0 +A_1$, we define the $K$-functional  by setting for $t>0$,
\[
K(x, t) =K\big(x,t; A_0,A_1\big)=\inf\Big\{ \big\|x_0  \big\|_{A_0 } + t\big\|x_1  \big\|_{ A_1}:\,  x=x_0 +x_1,\,  x_0 \in A_0,\,  x_1 \in A_1\Big\}.
\]
Note that  for  each $t>0$, $x \mapsto K(x,t)$ gives an equivalent  quasi-norm on $A_0  +A_1$.

If $0<\theta<1$ and $1\leq  \g<\infty$, the real interpolation space $A_{\theta, \g}=(A_0, A_1)_{\theta, \g}$  is defined by $x \in A_{\theta,\g}$ if and only if
\[
\big\| x \big\|_{(A_0,A_1)_{\theta, \g}} =\Big( \int_0^\infty \big(t^{-\theta}K\big(x, t; A_0,A_1\big)\big)^{\g }\ \frac{dt}{t} \Big)^{1/\g} <\infty.
\]
If $\g=\infty$, we define $ x \in A_{\theta,\infty}$ if and only if 
\[
\big\| x\big\|_{(A_0, A_1)_{\theta, \infty}}= \sup_{t>0} t^{-\theta} K(x, t; A_0, A_1)<\infty.
\]
For $0<\theta<1$ and $0< \g\leq \infty$, the functional $\|\cdot\|_{\theta,\g}$ is a  quasi-norm. In  the case where $A_0$ and $A_1$ are Banach spaces and $1\leq \g\leq \infty$,   $(A_{\theta,\g}, \|\cdot\|_{\theta,\g})$  can be renormed to be a Banach space. Moreover,
the space $A_{\theta,\g}$ is an interpolation space for the couple $(A_0, A_1)$  in the sense of Definition~\ref{def-interpolation}. There is also an equivalent description of $A_{\theta,\gamma}$  using a   dual notion called  $J$-functionals but this will not be needed  for our purpose below. Our main references for interpolations are the books \cite{BENSHA} and \cite{BL}.

\smallskip

It is worth mentioning   that the  real interpolation method is well understood for the couple $(L_{p_0}, L_{p_1})$ for both   the classical case and the  noncommutative case. We record here that
Lorentz spaces can be realized as  real interpolation spaces  for  the couple $(L_{p_0}, L_{p_1})$. More precisely, 
 if $\cal N$  is a semifinite von Neumann algebra, $1 \leq p_0<p_1\leq \infty$, $0<\theta<1$, and $1\leq q\leq \infty$ then,   up to equivalent  quasi-norms (independent of $\cal N$),  
\[
\big(L_{p_0}(\cal N), L_{p_1}(\cal N)\big)_{\theta,q}= L_{p,q}(\cal N)
\]
where $1/p=(1-\theta)/p_0 +\theta/p_1$. By reiteration,  if $1\leq \lambda, \g\leq \infty$, we also have 
\begin{equation}\label{Lp}
\big(L_{p_0, \lambda}(\cal N), L_{p_1, \g}(\cal N)\big)_{\theta,q}= L_{p,q}(\cal N)
\end{equation}
 with equivalent  quasi-norms when  $1 \leq p_0<p_1\leq \infty$, $0<\theta<1$,  $1\leq q\leq \infty$, and $1/p=(1-\theta)/p_0 +\theta/p_1$. These facts can be found in \cite{PX3} and will be used repeatedly throughout.

\smallskip

Wolff's interpolation theorem will be needed in  the next subsection. We record it here for convenience.
\begin{theorem}[{\cite[Theorem~1]{Wolff}}]\label{W} Let  $B_i$ ($i=1,2,3,4$) be quasi-Banach spaces such that $B_1 \cap B_4$ is dense in $B_j$ ($j=2,3$) and  satisfy:
\[
B_2=(B_1, B_3)_{\phi, r} \ \ \text{and}\ \ B_3=(B_2, B_4)_{\theta, q}
\]
for $0<\phi,\theta<1$ and  $0<r,q\leq \infty$. Then 
\[
B_2=(B_1, B_4)_{\xi,r} \ \ \text{and}\ \  B_3=(B_1, B_4)_{\zeta, q}
\]
where $\displaystyle{\xi=\frac{\phi\theta}{1-\phi +\phi\theta}}$ and $\displaystyle{\zeta=\frac{\theta}{1-\phi+\phi\theta}}$.
\end{theorem}
It is more convenient to  apply the Wolff's interpolation theorem using intervals.  
\begin{definition} 
A family of  quasi-Banach spaces $(A_{p,\g})_{p,\g \in (0,\infty]}$ is said to form a \emph{real interpolation scale} on an interval $I \subseteq \mathbb{R}\cup \{\infty\}$ if for every $p,q \in I$, $0<\g_1,\g_2,\g \leq \infty$, $0<\theta<1$, and $1/r=(1-\theta)/p + \theta/q$,
\[
A_{r, \g}= (A_{p,\g_1}, A_{q,\g_2})_{\theta,\g}
\]
with equivalent quasi-norms.
\end{definition}

The next lemma is a   version of Wolff's interpolation theorem at the level of family of real interpolation scale.

\begin{lemma}[{\cite[Lemma~3.4]{Ran-Int}}]\label{union}
 Assume that a family of quasi-Banach spaces $\mathfrak{F}=(A_{p,\g})_{p,\g \in (0,\infty]}$  forms a \emph{real interpolation scale} on  two  different intervals $I $ and $J$.  If  $|J \cap I |>1$,  then 
$\mathfrak{F}$  forms a real interpolation scale on  the interval $I\cup J$.
\end{lemma}

\smallskip

A  more  general  real interpolation  type spaces will be essential  in our  consideration below. Recall that a Banach function space $\cal{F}$ has a monotone norm if  whenever $f,g \in  \cal{F}$, $|f| \leq |g| \implies \|f\|_{\cal{F}} \leq  \|g\|_{\cal{F}}$.
\begin{definition}
  An  interpolation space $E$  for a  Banach couple  $(E_0,E_1)$ is said to be \emph{given by a $K$-method} if there exists a Banach function space $\cal{F}$  with monotone norm  such that $x \in E$ if and only if $t\mapsto K(x,t ; E_0,E_1) \in \cal{F}$ 
 and there exists a  constant $C_E>0$  such that
 \[
 C_E^{-1} \big\| t\mapsto K(x,t ; E_0,E_1)\big\|_{\cal{F}} \leq \big\|x\big\|_E \leq C_E \big\| t\mapsto K(x,t ; E_0,E_1)\big\|_{\cal{F}}.
 \]
In this case, we write $E=(E_0, E_1)_{\cal{F};K}$. 
\end{definition}

The following fact will be used in the sequel. This is known as a result of Brudnyi and Krugliak (see \cite[Theorem~6.3]{KaltonSMS}). 
 
 \begin{proposition}\label{K-method}
   Let  $1\leq p<q\leq \infty$. Every interpolation space   $E\in {\rm Int}(L_p,L_q)$ is given by a $K$-method.
  \end{proposition}
  
  In the next section, we  will sparingly use the concept of $K$-closed couple which we refer to \cite{KiX,Pisier-int,Ran-Int2} for formal definition and relevant properties.

\medskip

We end this subsection with a discussion on an extrapolation result that will be needed in the next section. 
It  involves the  notion of  Ces\`{a}ro  operator together with  its formal dual.  Recall that the  \emph{Ces\`{a}ro operator} $C:(L_1+L_\infty)(0,\infty)\to (L_{1,\infty}+L_{\infty})(0,\infty)$ is defined by the formula
\[
(Cf)(t):=\frac 1t\int_0^tf(s)\, ds,\quad f\in (L_1+L_\infty)(0,\infty).
\]
 The operator $C^{\ast}$ which we  will refer to as \emph{the dual Ces\`aro operator} is defined by setting
\[
(C^{\ast}f)(t):=\int_t^\infty\frac{f(s)}{s}\, ds,\quad f\in\Lambda_{\log}(0,\infty).
\]
It is known that $C^{\ast}$ is  bounded  from $\Lambda_{\log}(0,\infty)$ into $(L_1+L_{\infty})(0,\infty)$.
A proof of this fact can be found in \cite[Fact~2.9]{Jiao-Sukochev-Wu-Zanin}.
We refer to $C^{\ast}$ as the  (formal) {\it dual of  the Ces\`{a}ro operator} due to the fact that
\[\langle Cf,g\rangle=\langle f,C^{\ast}g\rangle,\quad f,g\in L_2(0,\infty).
\]
We will need  the following  well known results on boundedness of   Ces\`{a}ro  operators (see \cite[Theorem 327]{Hardy-Lit-Polya} for the detailed proof).

\begin{lemma}\label{C-bounded} 
The operators $C$ and  $C^{\ast}$ satisfy the following properties:
\[
\big\|C\big\|_{L_p \to L_p}=p', \quad 1<p\leq\infty,
\]
and
\[
\big\|C^{\ast}\big\|_{L_p \to L_p}=p, \quad 1 \leq p<\infty,
\]
where $p'$ is  the conjugate index of $p$.
\end{lemma}

We now consider   an extrapolation result from \cite{Jiao-Sukochev-Wu-Zanin}.  Assume that  $(\N_1,\nu_1)$ and $(\N_2,\nu_2)$ are  semifinite von Neumann algebras and suppose that $T: L_2(\N_1,\nu_1) \to L_2(\N_2,\nu_2)$ is a linear operator  that admits  bounded extensions  from $L_p(\N_1,\nu_1)$ to $L_p(\N_2,\nu_2)$ for all $2<p<\infty$ and satisfies further that
\[
\big\| T\big\|_{L_p(\cal{N}_1) \to L_p(\cal{N}_2)} \leq p, \quad 2<p<\infty.
\]
Such assumption is referred to  in \cite{Jiao-Sukochev-Wu-Zanin} as the  \emph{First Extrapolation Condition}.   The following  extrapolation will be used to deduce Proposition~\ref{distribution} below.
\begin{theorem}[{\cite[Theorem~3.2]{Jiao-Sukochev-Wu-Zanin}}]\label{extrapolation}
Suppose that $T$ satisfies the First Extrapolation condition. Then $T$ admits a bounded linear extension  from $(\Lambda_{\log}\cap (L_2+L_\infty))(\N_1,\nu_1)$ into $(L_2 +L_\infty)(\N_2,\nu_2)$. Moreover, we have
\[
\mu^2(Tx) \prec\prec c_{abs}\big( C^*\mu(x)\big)^2, \quad x \in (\Lambda_{\log}\cap (L_2+L_\infty))(\N_1,\nu_1).
\]
\end{theorem}

The following variant of the dual Doob inequality  is essential in the our proof  in next section. 

\begin{proposition}\label{distribution}
Let $(e_k)_{k\geq 1}$ be  a mutually disjoint sequence of projections in $\M$.   If $a$ is a positive operator so that $a^{1/2}  \in \big(\Lambda_{\log} \cap( L_2 + L_\infty)\big)(\M)$,  then
\[
\mu\Big( \sum_{k\geq 1} \E_{k-1}(e_ka e_k) \Big) \prec\prec c_{abs}\,  \Big(C^*[\mu^{1/2}(a)] \Big)^2.
\]
\end{proposition}
In preparation of the proof, we recall a crucial result established by Junge in \cite[Proposition~2.8]{Ju} (see also \cite{Chen-Ran-Xu} for the nonseparable case).

\begin{proposition}\label{module}
Let $\cal{N}$ be a semifinite von Neuman subalgebra of $\M$. If  $\E: \M \to \cal{N}$ is the trace preserving  conditional expectation, then there exists  an  $\cal{N}$-module  linear isometry $u: L_2(\M) \to L_2(\cal{N} \overline{\otimes} \cal{B}(\ell_2))$ such that
\[
u(x)^*u(y)=\E(x^*y) \otimes e_{1,1},  \quad x,y \in L_2(\M).
\]
\end{proposition}

\begin{proof}[Proof of Proposition~\ref{distribution}]
Let $\Theta: L_2(\M) \to L_2(\M \overline{\otimes} \cal{B}(\ell^2))$  be defined by
\[
\Theta(x)=\sum_{k\geq 1} xe_k \otimes e_{k,1}.
\]
Then, for  $2\leq p<\infty$ and $x\in L_p(\M)$,   we have
\begin{align*}
\big\|\Theta(x)\big\|_p^p &=\big\|\big(\sum_{k\geq 1} e_k|x|^2e_k)^{1/2}\big\|_p^p\\
&=\big\|\sum_{k\geq 1} e_k|x|^2e_k\big\|_{p/2}^{p/2}\\
&\leq \big\| |x|^2 \|_{p/2}^{p/2}=\big\|x\big\|_p^p.
\end{align*}
That is, $\big\|\Theta\big\|_{L_p\to L_p} \leq 1$. Next, for  $k\geq 1$, let $u_k: L_2(\M) \to L_2(\M_k \overline{\otimes} \cal{B}(\ell_2)) $ be  the $\M_k$-module map  as in Proposition~\ref{module}  that satisfies
\[
|u_k(x)|^2 =\E_k(x^*x) \otimes e_{1,1}, \quad  x \in L_2(\M).
\]
Consider  the mapping $T$ formally defined by  the formula
\[
Ta =\sum_{k\geq 1} u_{k-1}(a_{k1}) \otimes e_{k,1},\quad a=(a_{kj})_{k,j\geq 1}\in L_0(\M\overline{\otimes} \cal{B}(\ell_2)).
\]
One can easily see that $|T(a)|^2 =\sum_{k\geq 1} \E_{k-1}(|a_{k1}|^2) \otimes e_{1,1}\otimes e_{1,1}$.  By the dual Doob inequality  (\cite{Ju, JX2}), we have for $2\leq p<\infty$,
\begin{align*}
\big\|Ta\|_{p}^p &=\big\| |Ta|^2\big\|_{p/2}^{p/2}\\
 &\leq (c_{abs} (p/2)^2)^{p/2} \big\| \sum_{k\geq  1} |a_{k1}|^2 \big\|_{p/2}^{p/2}\\
 &\leq c_{abs}^{p/2} p^p \big\|a\big\|_p^p.
\end{align*}
We obtain that $\|T\|_{L_p \to L_p} \leq c_{abs}^{1/2} p $. This shows  in particular that the operator \[
T\circ\Theta: L_2(\M) \to L_2(\M \overline\otimes \cal{B}(\ell_2) \overline\otimes \cal{B}(\ell_2))
\] satisfies the First Extrapolation Condition.  Applying Theorem~\ref{extrapolation}, we have for every  $x  \in (\Lambda_{\log}\cap (L_2+L_\infty))(\M)$,
\[
\mu^2(T\circ\Theta(x)) \prec\prec c_{abs}\big( C^*\mu(x)\big)^2.
\]
Note that $\mu^2(T\circ\Theta(x))=\mu(|T\circ\Theta(x)|^2)$ and $|T\circ\Theta(x)|^2 =\sum_{k\geq 1}\E_{k-1}(e_k|x|^2e_k) \otimes e_{1,1} \otimes e_{1,1}$.  The desired conclusion  follows by applying the above  submajorization to  $x=a^{1/2}$.
\end{proof}

\section{Interpolations of  martingale Hardy spaces}

\subsection{An estimate of the $K$-functional of the couple $(\H_2^c, \H_\infty^c)$}

The following is the  primary result of the paper.
\begin{theorem}\label{main}
 For every $x \in  \H_{\Lambda_{\log}}^c(\M) \cap (\H_2^c(\M) +\H_\infty^c(\M))$ and $t>0$, the following inequality holds:
\[
K\big(x, t; \H_2^c(\M), \H_\infty^c(\M) \big)\leq C_{abs} \left(\int_0^{t^2} \Big(\mu_u(S_c(x))+ C^{\ast}[\mu(S_c(x)](u)\Big)^2 \, du\right)^{1/2}
\]
where $C^{\ast}$ is the dual Ces\`aro operator. 
\end{theorem}

\begin{proof}
The    argument  below  follows the  strategy  used for   the conditioned case in  \cite[Proposition~3.2]{Ran-Int2}. The main  new idea is the incorporation of  the dual Ces\`aro  operator $C^*$.  We include all crucial details for completeness.

Assume that  $x \in  \H_{\Lambda_{\log}}^c(\M) \cap (\H_2^c(\M) + \H_\infty^c(\M))$. This is equivalent to  the condition that $S_c(x) \in \Lambda_{\log}(\M) \cap( L_2 + L_\infty)(\M)$ or $\mu(S_c(x)) \in \Lambda_{\log}(0,\infty) \cap (L_2 + L_\infty)(0,\infty)$.  In particular, $C^*[\mu(S_c(x))]$ is well-defined and belongs to $(L_2 +L_\infty)(0,\infty)$.  We will construct a concrete decomposition of $x$ that  will provide the desired estimate on the $K$-functional.

Fix $t>0$ and set a parameter
\begin{equation}\label{lambda}
\lambda := \frac{3}{t}\Big(  \int_0^{t^2} \mu_u(S_c^2(x))\ du +\kappa\int_0^{t^2} \big(C^{\ast}[\mu(S_c(x))](u)\big)^2 \ du\Big)^{1/2}
\end{equation}
where $\kappa$ denotes the absolute constant from Proposition~\ref{distribution}. Although at first this choice of $\lambda$   seems  artificial, it will become more transparent  during the course of the proof.

As in \cite{Ran-Int2}, the argument  requires two steps.

$\bullet$ Step~1. 

 We apply the construction of Cuculescu's projections to the submartingale $(S_{c, k}^2 (x))_{k \ge 1}$  and the  parameter $\lambda^2$. That is,  we start with  $q_0 ={\bf 1}$ and for $k\geq 1$, we set
\[
q_k := q_{k-1} \ch_{[0,\lambda^2]} \big(
q_{k-1}  S_{c,k}^2(x) q_{k-1} \big)=\ch_{[0,\lambda^2]}\big(
q_{k-1} S_{c,k}^2(x)  q_{k-1} \big)q_{k-1}.
\]
Then   $(q_k)_{k\ge 1}$ is a decreasing sequence  of projections  in $\M$ satisfying the following properties:
\begin{enumerate}
\item $q_k \in \M_{k}$ for every $k\geq 1$;
\item  $q_k$ commutes with $q_{k-1} S_{c, k}^2 (x) q_{k-1}$ for all $k\geq 1$;
\item $q_k S_{c, k}^2 (x) q_k \le \lambda^2 q_k$ for all $k\geq 1$;
\item  if we set $q= \bigwedge_{k \ge 1} q_k$,  then 
$
\lambda^2({\bf 1}- q)  \leq \sum_{k\geq 1} (q_{k-1}-q_{k}) S_{c,k}^2(x) (q_{k-1}-q_k)$.
 \end{enumerate}
Verifications of these  facts  concerning  Cuculescu's projections  can be found in \cite{Chen-Ran-Xu, Cuc, PR, Ran15}.
 
 Using the sequence $(q_k)_{k\geq 0}$, we  consider the following  adapted   sequence $\a=(\a_k)_{k\geq 1}$ by setting:
\begin{equation}\label{a}
 \alpha_k = d x_k q_k, \quad k\geq 1.
\end{equation}
Next, we modify  $\a$ into a martingale difference sequence by setting:
\begin{equation}\label{b}
d\beta_k = \a_k -\E_{k-1}(\a_k), \quad k\geq 1.
\end{equation}
We denote  the corresponding   martingale  by  $\b=(\b_k)_{k\geq 1}$.  The main difference with the conditioned case is that $(q_k)_{k\geq 0}$ is only an adapted sequence as oppose to  being predictable in  \cite{Ran-Int2} and thus the adjustment taken  in the definition of $\b$.

We record the next lemma for further use.
\begin{lemma}\label{lemma1}
 We have the following properties:
\begin{enumerate}[$\rm (i)$]
\item $\sup_{k\geq 1}\|\a_k\|_\infty \le \lambda$;
\item $\sup_{k\geq 1} \|d\beta_k\|_\infty \le 2\lambda$;
\item   $\cal{S}_c^2(\a)  \prec\prec  4 S_c^2(x)$;
\item $\mu(S_c^2(\b)) \prec\prec 8\mu(S_c^2(x)) + 2\kappa\Big(C^*\big[\mu(S_c(x))\big]\Big)^2$.
\end{enumerate}
\end{lemma}
\begin{proof}
The first item   follows easily from the construction. Indeed,  given $k\geq 1$, 
\begin{equation*}\begin{split}
|\a_k|^2  &= q_k | d x_k |^2 q_k \\
&=q_k[ S_{c,k}^2(x) -S_{c,k-1}^2(x)]q_k\\
&\leq q_k S^2_{c, k} (x) q_k
\leq \lambda^2 q_k.
\end{split}
\end{equation*}
The second item clearly follows from the fact that conditional expectations are contractive projections on $\M$.
For the third item,  we consider first the martingale difference sequence 
\[
d\g_k =dx_kq_{k-1},\quad k\geq 1.
\]
We have from the definition  of $d\g$ that  for $m\geq 1$  (with $S_{c,0}(x)=0$),
\begin{align*}
{S}_{c,m}^2(\g) &= \sum_{k=1}^m q_{k-1} |dx_k|^2 q_{k-1}\\
&=\sum_{k=1}^m q_{k-1} [S_{c,k}^2(x)- S_{c,k-1}^2(x)]q_{k-1}\\
&=\sum_{k=1}^m q_{k-1}S_{c,k}^2(x)q_{k-1} - \sum_{k=1}^m q_{k-1} S_{c,k-1}^2(x)q_{k-1}.
\end{align*}
Performing some indexing shift, we obtain that
\[
S_{c,m}^2(\g)= q_{m-1} S_{c,m}^2(x) q_{m-1} +\sum_{k=1}^{m-1} \big(q_{k-1} S_{c,k}^2(x)q_{k-1}- q_{k} S_{c,k}^2(x)q_{k}\big).
\]
  From the fact  that $q_k$ commutes with $q_{k-1} S_{c, k}^2 (x) q_{k-1}$,  we deduce that
\begin{align*}
S_{c,m}^2(\g) &=q_{m-1} S_{c,m}^2(x)  q_{m-1}+\sum_{k=1}^{m-1} (q_{k-1}-q_k)S_{c,k}^2(x) (q_{k-1}-q_k)\\
&\leq q_{m-1} S_{c}^2(x) q_{m-1}+\sum_{k=1}^{m-1} (q_{k-1}-q_k)S_c^2(x) (q_{k-1}-q_k).
\end{align*}
Note that the finite family of projections $\{q_{k-1}-q_k: 1\leq k\leq m-1\} \cup \{q_{m-1}\}$  is mutually disjoint. We may  deduce  from \eqref{sub-diagonal}  that  for every $m\geq 1$, $S_{c,m}^2(\g) \prec\prec S_c^2(x)$. Next, for every  $w>0$, the monotone convergence theorem gives:
\[
\int_0^w \mu_u(S_c^2(\g))\ du =\lim_{m\to \infty} \int_0^w \mu_u(S_{c,m}^2(\g))\ du \leq 
\int_0^w \mu_u(S_c^2(x))\ du. 
\]
That is, $S_{c}^2(\g) \prec\prec S_c^2(x)$.

On the other hand, a simple computation  gives:
\begin{align*}
\cal{S}_c^2(d\g-\a) &=\sum_{k\geq 1} (q_{k-1}-q_k) [S_{c,k}^2(x)-S_{c,k-1}^2(x)] (q_{k-1}-q_k) \\ 
&\leq  \sum_{k\geq 1} (q_{k-1}-q_k)S_c^2(x) (q_{k-1}-q_k)
\end{align*}
 and therefore, we also have from \eqref{sub-diagonal} that  $\cal{S}_c^2(d\g-\a) \prec\prec S_c^2(x)$.  Using the elementary inequality $|a +b|^2 \leq 2|a|^2 + 2|b|^2$ for operators $a$ and $b$, we can   conclude  that
 \[
 \cal{S}_c^2(\a) \leq 2\cal{S}_c^2(d\g-\a) +2S_c^2(\g) \prec\prec  4 S_c^2(x)
 \]
which is item~(iii). 

For the last item,  we  begin with the  simple fact that 
\[
S_c^2(\beta)  \leq  2 \cal{S}_c^2 (\a) + 2\sum_{k\geq 1} | \E_{k-1}(\a_k)|^2.
\]
Recall that $\a_k=dx_kq_k= dx_k(q_k-q_{k-1}) + dx_k q_{k-1}$. Since $(dx_kq_{k-1})_{k\geq 1}$ is a martingale difference sequence, we have $\E_{k-1}(\a_k)=\E_{k-1}(dx_k(q_k-q_{k-1}))$. Applying  Kadison's inequality $|\E_{k-1}(a)|^2 \leq \E_{k-1}(|a|^2)$ for any operator $a$ and $k\geq 1$, we further obtain that
\begin{align*}
S_c^2(\beta) &\leq  2 \cal{S}_c^2(\a) + 2\sum_{k\geq 1} \E_{k-1}\big((q_{k-1}-q_k)|dx_k|^2 (q_{k-1}-q_k)\big)\\
&\leq 2 \cal{S}_c^2(\a) + 2\sum_{k\geq 1} \E_{k-1}\big[(q_{k-1}-q_k)S_c^2(x) (q_{k-1}-q_k)\big].
\end{align*}
It follows from \eqref{sub-sum}  and Item~(iii) that
\begin{align*}
\mu\big(S_c^2(\beta)\big) &\prec\prec  2\mu\big(\cal{S}_c^2(\a) \big) +2 \mu\big(\sum_{k\geq 1} \E_{k-1}\big[(q_{k-1}-q_k)S_c^2(x) (q_{k-1}-q_k)\big]\big)\\
&\prec\prec  8\mu\big({S}_c^2(x) \big) +2 \mu\big(\sum_{k\geq 1} \E_{k-1}\big[(q_{k-1}-q_k)S_c^2(x) (q_{k-1}-q_k)\big]\big).
\end{align*}
Next,  we apply   Proposition~\ref{distribution} to the mutually disjoint sequence $\{(q_{k-1}-q_k)\}$ and the operator  $S_c^2(x)$.  Note that $S_c(x)\in (\Lambda_{\log} \cap(L_2 +L_\infty))(\M)$  and therefore it satisfies the assumption used in Proposition~\ref{distribution}.  We  obtain that 
\[ 
\mu(S_c^2(\beta))  \prec \prec   8 \mu(S_c^2 (x)) + 2\kappa\Big(C^*\big[\mu(S_c(x))\big]\Big)^2\]
which is the desired submajorization.
\end{proof}

$\bullet$  Step~2. Construction of  a  decomposition that will provide    the stated estimate on the $K$-functional.

As in the first step, we  apply the construction of  Cuculescu's projections to  the submartingale $(S_{c, k}^2 (\b))_{k \ge 1}$ with  the parameter $\lambda^2$ where $\b$ is the martingale from \eqref{b}.  
That is,  setting   $\pi_0 ={\bf 1}$ and for $k\geq 1$, we define:
\[
\pi_k := \pi_{k-1} \ch_{[0,\lambda^2]} \big(
\pi_{k-1}  S_{c,k}^2(\b) \pi_{k-1} \big)=\ch_{[0,\lambda^2]}\big(
\pi_{k-1} S_{c,k}^2(\b)  \pi_{k-1} \big)\pi_{k-1}.
\]
Then,   $(\pi_k)_{k\ge 1}$ is a decreasing sequence  of projections  in $\M$. As before, it satisfies the following properties:
\begin{enumerate}
\item $\pi_k \in \M_{k}$ for every $k\geq 1$;
\item  $\pi_k$ commutes with $\pi_{k-1} S_{c, k}^2 (\b) \pi_{k-1}$ for all $k\geq 1$;
\item $\pi_k S_{c, k}^2 (\b) \pi_k \le \lambda^2 \pi_k$ for all $k\geq 1$;
\item  if we set $\pi= \bigwedge_{k \ge 1} \pi_k$,  then 
$
\lambda^2({\bf 1}- \pi)  \leq \sum_{k\geq 1} (\pi_{k-1}-\pi_{k}) S_{c,k}^2(\b) (\pi_{k-1}-\pi_k)$.
 \end{enumerate}

 Next, we define two martingales $y$ and  $z$ by setting:
\begin{equation}\label{z}
z= \sum_{k \geq 1} d \b_k \pi_{k - 1}=\sum_{k\geq 1}[dx_kq_k-\E_{k-1}(dx_kq_k)]\pi_{k-1} \ \ \text{and} \ \ 
y=x-z.
\end{equation}

We will show that this decomposition provides the desired estimate on the $K$-functional.
We  consider first the martingale $z$. We claim that   $z \in \H_\infty^c(\M)$ with 
 \begin{equation}\label{norm-z}
 \big\|z\big\|_{\H_\infty^c} \leq \sqrt{5}  \lambda.
\end{equation}
To verify \eqref{norm-z},  we first fix $m\geq 1$ and estimate $S_{c,m}(z)$.  From the definition of $z$, we have:
\begin{align*}
S_{c,m}^2 (z)& = \sum_{k=1}^m \pi_{k-1} |d \b_k|^2 \pi_{k-1}\\
& = \sum_{k= 1}^m\big ( \pi_{k-1} S^2_{c,k} (\b) \pi_{k-1} - \pi_{k-1} S^2_{c,k-1} (\b) \pi_{k-1} \big )\\
& =  \sum_{k=1}^m   \pi_{k-1} S^2_{c,k} (\b) \pi_{k-1} - \sum_{k=1}^{m-1}\pi_k S^2_{c,k} (\b) \pi_k \\
&=\pi_{m-1}S_{c,m}^2(\b) \pi_{m-1} + \sum_{k=1}^{m-1} ( \pi_{k-1} - \pi_k ) S^2_{c,k} (\b) ( \pi_{k-1} - \pi_k \big )\,,
\end{align*}
where the last equality follows from the commutativity between $\pi_k$ and $\pi_{k-1} S^2_{c,k} (\b) \pi_{k-1}$.  Recall from Lemma~\ref{lemma1}~(ii) that $\| d\b_k\|_{\infty} \le 2 \lambda$.  Using this fact,  we have
\begin{align*} 
\pi_{k-1} S_{c,k}^2(\b)\pi_{k-1} &= \pi_{k-1}[S_{c,k-1}^2(\b) +|d \b_k|^2)]\pi_{k-1}\\
&\leq 5\lambda^2 \pi_{k-1}.
\end{align*}
We can deduce  that for every $m\geq 1$, 
\[
S_{c,m}^2 (z) 
\leq \lambda^2 \pi_{m-1} + 5\lambda^2 \sum_{k=1}^{m-1} ( \pi_{k-1} - \pi_k)
\leq 5\lambda^2{\bf 1}.
\]
Since this  holds for arbitrary $m\geq 1$, we have $S_{c}^2 (z)\leq  5\lambda^2{\bf 1}$
which   shows  that $\big\|z\big\|_{\H_\infty^c} \leq  \sqrt{5} \lambda$ and thus proving inequality \eqref{norm-z}.

 \smallskip

We now deal with the martingale $y$.  We will estimate the norm of  $y$ in $\H_2^c(\M)$.  The verification of the next lemma is the most  delicate part of the argument. We take the opportunity to point out that the proof below together with the submajorization in Lemma~\ref{lemma1}(iv) motivated the choice  of $\lambda$  taken in \eqref{lambda}.
\begin{lemma}\label{trace}
The projections $q$ and $\pi$ satisfy the following property:
\[
\max\big\{ \T({\bf 1}-q), \T({\bf 1}-\pi) \big\} \leq t^2.
\]
\end{lemma}
\begin{proof}
Fix $w>t^2$. We will show that    $\mu_w({\bf 1}-q)= \mu_w({\bf 1}-\pi)=0$.
We provide the argument for ${\bf 1}-\pi$. The case of ${\bf 1}-q$ is simpler since it does not depend on  the second step.

Assume the opposite, i.e, $\mu_w({\bf 1}-\pi)=1$.
We start with  the fact that
\begin{equation}\label{pi}
\lambda^2({\bf 1}- \pi) \leq \sum_{k\geq 1} (\pi_{k-1}-\pi_{k}) S_{c}^2(\b) (\pi_{k-1}-\pi_k).
\end{equation}
Taking generalized singular values and integrals,  inequality \eqref{pi} gives
\[
\lambda^2 \int_0^w \mu_u({\bf 1}-\pi) \, du \leq \int_0^w \mu_u\big(\sum_{k\geq 1} (\pi_{k-1}-\pi_{k}) S_{c}^2(\b) (\pi_{k-1}-\pi_k)\big) \, du
\]
By submajorization and the fact that $\mu({\bf 1}-\pi)$ is a characteristic function and therefore is identically equal to $1$ on the interval $[0,w]$ by assumption, we have 
\[
\lambda^2 w  \leq \int_0^w \mu_u(S_c^2(\b))\ du \leq 8  \int_0^w\mu_u(S_c^2 (x)) \ du + 2\kappa\int_0^w \Big(C^*\big[\mu(S_c(x))\big](u) \Big)^2 \ du\]
where the second inequality comes from  the submajorization in Lemma~\ref{lemma1}(iv). Using the specific value of $\lambda$ in \eqref{lambda}, this  leads  to 
\begin{align*}
9w\int_0^{t^2}\mu_u(S_c^2 (x)) \ du &+ 9w\kappa\int_0^{t^2} \Big(C^*\big[\mu(S_c(x))\big](u) \Big)^2 \ du \\
&\leq 8t^2 \int_0^w\mu_u(S_c^2 (x)) \ du  + 2t^2\kappa \int_0^w \Big(C^*\big[\mu(S_c(x))\big](u) \Big)^2 \ du \\
&= I + II.
\end{align*}
Next, we estimate $I$ and $II$ separately. For $I$, we have:
\begin{align*}
I &=8t^2 \int_0^{t^2}\mu_u(S_c^2 (x)) \ du + 8t^2\int_{t^2}^w\mu_u(S_c^2 (x)) \ du  \\
&\leq  9t^2\int_0^{t^2}\mu_u(S_c^2 (x)) \ du + 8t^2(w-t^2)\mu_{t^2}(S_c^2(x)).
\end{align*}
Similarly, $II$ can be estimated as follows:

\begin{align*}
II &= 2t^2 \kappa \int_0^{t^2} \Big(C^*\big[\mu(S_c(x))\big](u)\Big)^2 \, du  +2t^2 \kappa \int_{t^2}^w \Big(C^*\big[\mu(S_c(x))\big](u)\Big)^2 \, du\\
  &\leq 9t^2 \kappa \int_0^{t^2} \Big(C^*\big[\mu(S_c(x))\big](u)\Big)^2 \, du  +2t^2 \kappa \int_{t^2}^w \Big(C^*\big[\mu(S_c(x))\big](u)\Big)^2 \, du\\
&\leq 9t^2 \kappa \int_0^{t^2} \Big(C^*\big[\mu(S_c(x))\big](u)\Big)^2 \, du  +2t^2 \kappa (w-t^2)\Big(C^*\big[\mu(S_c(x))\big](t^2)\Big)^2 
\end{align*}
where in the last estimate we have use the fact that  the function $C^*[\mu(S_c(x))]$ is decreasing.
 Using these estimates on $I$ and $II$, we obtain after rearrangement  and division  by $w-t^2$ that
\[
9 \int_0^{t^2} \mu_u(S_c^2 (x)) \ du  + 9\kappa \int_0^{t^2}\Big(C^*\big[\mu(S_c(x))\big](u)\Big)^2  \, du  \leq  8t^2\mu_{t^2}(S_c^2(x)) +2\kappa t^2 \big(C^*\big[\mu(S_c(x))\big](t^2)\big)^2.
\]
But  the left hand side of the preceding inequality  is larger than the quantity
$ 9t^2\mu_{t^2}(S_c^2(x)) +9\kappa t^2 \big(C^*\big[\mu(S_c(x))\big](t^2)\big)^2$
 which  is a contradiction. Thus, we may conclude that
$\mu_w({\bf 1}-\pi)=0$. This shows  in particular   that $\T({\bf 1}-\pi)\leq t^2$.
The proof for $\T({\bf 1}-q)$ is identical.
\end{proof}
Now we proceed with the estimation of the $\H_2^c$-norm of $y$. We begin by observing that for $k\geq 1$, it follows from the definitions $z$ and $y$ that $dy_k$ can be split into  three separate parts:  
\begin{align*}
dy_k &= dx_k-d\b_k\pi_{k-1}\\
&= dx_k({\bf 1}-\pi_{k-1}) + dx_k({\bf 1}-q_k)\pi_{k-1}  +\E_{k-1}(dx_k q_k)\pi_{k-1}\\
&=dx_k({\bf 1}-\pi_{k-1}) + dx_k({\bf 1}-q_k)\pi_{k-1}  +\E_{k-1}[dx_k( q_k-q_{k-1})]\pi_{k-1}.
\end{align*}
Using the facts that  conditional expectations  are contractive projections on $L_2$ and the functional   $\|\cdot\|_2^2$ is convex, we have for each $k\geq 1$,
\begin{align*}
\|dy_k\|_2^2 &\leq 4 \|dx_k({\bf 1}-\pi_{k-1})\|_2^2 + 4 \|dx_k({\bf 1}-q_k)\pi_{k-1} \|_2^2 +2\| \E_{k-1}[dx_k( q_k-q_{k-1})]\pi_{k-1}\|_2^2\\
&\leq  4 \|dx_k({\bf 1}-\pi_{k-1})\|_2^2 + 4 \|dx_k({\bf 1}-q_k) \|_2^2 + 2\|dx_k(q_{k-1}-q_k)\|_2^2\\
&\leq 4\|dx_k({\bf 1}-\pi)\|_2^2 + 6\|dx_k({\bf 1}-q)\|_2^2.
\end{align*}
Taking summation over $k$, we deduce the following estimate:
\begin{align*} 
\|y\|_{\H_2^c}^2 
&=\sum_{k\geq 1} \|dy_k\|_2^2\\
&\leq  4\sum_{k\geq 1}  \|dx_k({\bf 1}-\pi)\|_2^2 + 6 \sum_{k\geq 1} \|dx_k({\bf 1}-q)\|_2^2\\
&=  4\sum_{k\geq 1}  \T\big[({\bf 1}-\pi)|dx_k|^2({\bf 1}-\pi)\big] + 6 \sum_{k\geq 1}  \T\big[({\bf 1}-q)|dx_k|^2({\bf 1}-q)\big]\\
&=4\T\big[({\bf 1}-\pi)S_c^2(x)({\bf 1}-\pi)\big] +6 \T\big[({\bf 1}-q)S_c^2(x)({\bf 1}-q)\big].
\end{align*}
Moreover, it follows from Lemma~\ref{trace} and properties of generalized singular values that 
\begin{align*}
\big\|y \big\|_{\H_2^c}^2 &\leq 
4 \int_0^\infty \mu_u\big(({\bf 1}-\pi) S_c^2(x)({\bf 1}-\pi) \big)\, du  + 6 \int_0^\infty \mu_u\big(({\bf 1}-q) S_c^2(x)({\bf 1}-q) \big)\, du\\
&=4 \int_0^{t^2} \mu_u\big(({\bf 1}-\pi) S_c^2(x)({\bf 1}-\pi) \big)\, du  + 6 \int_0^{t^2} \mu_u\big(({\bf 1}-q) S_c^2(x)({\bf 1}-q) \big)\, du\\
&\leq 10 \int_0^{t^2} \mu_u\big( S_c^2(x) \big)\, du.
\end{align*}
Thus, we arrive at the inequality, 
\begin{equation}\label{norm-y}
\big\|y\big\|_{\H_2^c}\leq  \sqrt{10}   \Big( \int_0^{t^2} \mu_u(S_c^2(x))\, du \Big)^{1/2}.
\end{equation}

We  now estimate the $K$-functional using the decomposition $x=y+z$. By 
combining \eqref{norm-z} and \eqref{norm-y}, we have
\begin{align*}
K&\big(x, t; \H_2^c(\M), \H_\infty^c(\M)\big) \leq  \big\|y \big\|_{\H_2^c} + t \big\|z\big\|_{\H_\infty^c} \\
&\leq   \sqrt{10} \Big( \int_0^{t^2} \mu_u(S_c^2(x))\, du \Big)^{1/2} +
3\sqrt{5}\Big(  \int_0^{t^2} \mu_u(S_c^2(x))\ du +\kappa\int_0^{t^2} \big(C^{\ast}[\mu(S_c(x))](u)\big)^2 \ du\Big)^{1/2}.
\end{align*}
We  can now conclude that
\[
K\big(x, t; \H_2^c(\M), \H_\infty^c(\M)\big) \leq [\sqrt{10} +3\sqrt{5}(1+\kappa)^{1/2}]  \Big( \int_0^{t^2}    \big( \mu_u(S_c(x)) +C^*[\mu(S_c(x))](u)\big)^2 \, du\   \Big)^{1/2}. 
\]
The proof   is complete.
\end{proof}

For application purposes, it is important to view  the statement of Theorem~\ref{main} as  a comparison between  two different  $K$-functionals.  We  recall  that for $f \in L_2(0,\infty) +L_\infty(0,\infty)$ and $t>0$, we have from \cite{Holm} the following equivalence:
\[
K(f,t; L_2, L_\infty) \approx\Big( \int_0^{t^2} (\mu_u(f))^2 \, du \Big)^{1/2}.
\]
With this connection, we have the following reformulation of Theorem~\ref{main}.

\begin{remark}\label{main-reform}
For every $x \in  \H_{\Lambda_{\log}}^c(\M) \cap (\H_2^c(\M) +\H_\infty^c(\M))$ and $t>0$, the following inequality holds:
\[
K\big(x, t; \H_2^c(\M), \H_\infty^c(\M) \big)\leq C_{abs} K\big( \mu(S_c(x)) +C^{\ast}[\mu(S_c(x))], t; L_2, L_\infty\big).
\]
\end{remark}

At the time of this writing, we do not know if the use of the operator $C^{\ast}$ in  the estimate in Theorem~\ref{main} can be avoided. That is, it is unclear  if  for all $x \in \H_2^c(\M)+\H_\infty^c(\M)$, the equivalence $K(x, t;\H_2^c(\M),\H_\infty^c(\M))\approx K(S_c(x), t; L_2(\M),\M)$  holds as in the conditioned case treated in \cite{Ran-Int2}. In other words, it is still  an open question  if the couple $(\H_2^c(\M),\H_\infty^c(\M))$ is $K$-closed 
in the larger  couple $(L_2(\M \overline{\otimes}\cal{B}(\ell_2)), \M \overline{\otimes}\cal{B}(\ell_2))$.
Nevertheless, the estimate given in Theorem~\ref{main} and Remark~\ref{main-reform} is sufficient to deduce satisfactory  results concerning interpolations of  the couple $(\H_1^c(\M), \H_\infty^c(\M))$ as we  will explore in the next  subsection.


\subsection{Applications of Theorem~\ref{main}}


\subsubsection{A Peter Jones type interpolation  theorem}

Our result in this part  constitutes    the initial motivation for the paper. 
It fully resolved the real interpolations for the couple $\big( \H_{1}^c(\M), \H_\infty^c(\M)\big)$.
  It will be deduced from Theorem~\ref{main} and Wolff's interpolation theorem.

\begin{theorem}\label{main-interpolation} If $0<\theta <1$, $1/p=1-\theta$, and  $1\leq \g \leq \infty$, then 
\[
\big( \H_{1}^c(\M), \H_\infty^c(\M)\big)_{\theta, \g} = \H_{p,\g}^c(\M)
\]
with equivalent norms.
\end{theorem}

\begin{proof} We need two steps.

$\bullet$ {\it Step 1.} This concerns  the couple $(\H_2^c(\M), \H_\infty^c(\M))$. Let  $0<\theta<1$ and $1/q=(1-\theta)/2$.  Fix $y \in \big( \H_{2}^c(\M), \H_\infty^c(\M)\big)_{\theta, \g}$.  We verify first that
\begin{equation}\label{left}
\big\| y\big\|_{\H_{q,\g}^c} \lesssim \big\|y \big\|_{( \H_{2}^c(\M), \H_\infty^c(\M))_{\theta, \g} }.
\end{equation}
Indeed, 
since for every symmetric  function space $E$ on $(0,\infty)$, the space $\H_E^c(\M)$  embeds isometrically into $E(\M;\ell_2^c)$  by  the map $x \mapsto (dx_n)_{n\geq 1}$, we have 
\begin{align*}
\big\| y\big\|_{\H_{q,\g}^c} &=\big\| (dy_n) \big\|_{L_{q,\g}(\M;\ell_2^c)}\\
&\approx \big\| (dy_n) \big\|_{(L_{2}(\M;\ell_2^c), L_{\infty}(\M;\ell_2^c))_{\theta,\g}}\\
&\leq \big\| y \big\|_{( \H_{2}^c(\M), \H_\infty^c(\M))_{\theta, \g} }.
\end{align*}

\smallskip

For the reverse inequality, assume first that  $x\in \H_2^c(\M) +\H_\infty^c(\M)$ with $S_c(x)\in \Lambda_{\log}(\M)$.  By Theorem~\ref{main} (see also Remark~\ref{main-reform}), we have for $1\leq \g\leq \infty$,
\begin{align*}
\big\|x\big\|_{(\H_2^c(\M), \H_\infty^c(\M))_{\theta,\g}} &\lesssim \big\|  \mu(S_c(x)) + C^{\ast}[\mu(S_c(x))] \big\|_{(L_2,L_\infty)_{\theta,\g}}\\
&\approx_{q,\g} \big\| \mu(S_c(x)) +C^{\ast}[\mu(S_c(x))] \big\|_{q,\g}\\
&\lesssim_{p,\gamma} \big\| \mu(S_c(x))\big\|_{q,\g} +\big\|C^{\ast}[\mu(S_c(x))] \big\|_{q,\g}.
\end{align*}
The important fact here is that  since $2<q<\infty$,  we have   from  Lemma~\ref{C-bounded} and interpolation that  $C^{\ast}: L_{q,\g} \to L_{q,\g}$ is bounded. Therefore, we obtain further that
\begin{equation}\label{right}
\big\|x\big\|_{(\H_2^c(\M), \H_\infty^c(\M))_{\theta,\g}} \lesssim_{q,\g} \big\|\mu(S_c(x))\big\|_{q,\g}= \big\|x\big\|_{\H_{q,\g}^c}.
\end{equation}

Now we remove the  extra assumption that $S_c(x)\in \Lambda_{\log}(\M)$.
Consider an arbitrary  $y \in \H_{q,\g}^c(\M)$. Since $\M_1$ is semifinite, we  may fix an increasing sequence of projections $\{e_j\}$ in $\M_1$ with  $e_j \uparrow^j {\bf 1} $ and so that  for every $j\geq 1$, $\T(e_j)<\infty$. For $j\geq 1$,  we define the martingale 
\[
 y^{(j)}=  (y_ne_j)_{n\geq 1}.
\]
The sequence of martingales $(y^{(j)})_{j\geq 1}$ satisfies the following properties:
\begin{enumerate}[$\rm (i)$]
\item for every $j\geq 1$, $y^{(j)} \in \H_{\Lambda_{\log}}^c(\M) \cap \H_{q,\g}^c(\M)$;
\item $\lim_{j\to \infty } \big\| y^{(j)}-y\big\|_{\H_{q,\g}^c}=0$.
\end{enumerate}
We verify first that  for every $j\geq 1$, $y^{(j)} \in \H_{q,\g}^c(\M)$. Indeed, one easily sees that $S_c(y^{(j)})= (e_jS_c^2(y) e_j )^{1/2}=| S_c(y)e_j|$. It  then follows that 
\[\|y^{(j)}  \|_{\H_{q,\g}^c} =\|S_c(y^{(j)})\|_{q,\g} \leq \|S_c(y)\|_{q,\g}=\|y\|_{\H_{q,\g}^c} <\infty.
\]
For the $\Lambda_{\log}$-case, we have  from the definition that  for  any given $j\geq 1$,
\begin{align*}
\big\| y^{(j)}\big\|_{\H_{\Lambda_{\log}}^c} &= \int_0^\infty \frac{\mu_t\big(S_c(y)e_j\big)}{1+t}\, dt\\
&=\int_0^{\T(e_j)} \frac{\mu_t\big(S_c(y)e_j\big)}{1+t}\, dt\\
&\leq \int_0^{\T(e_j)} \mu_t\big(S_c(y)\big)\, dt\\
&\leq \big\| y\big\|_{\H_{q,\g}^c}  \big\| \ch_{[0, \T(e_j)]}(\cdot) \big\|_{q',\g'} <\infty
\end{align*}
where the next to last inequality comes from the fact that  $L_{q,\g}$ is the K\"othe dual of $L_{q',\g'}$. This verifies the first item. As a consequence, inequality  \eqref{right} applies to $y^{(j)}$ for every $j\geq 1$. 

The second item follows  at once from $S_c(y^{(j)}-y)=|S_c(y)({\bf 1}-e_j)|$. Now, using these two properties and \eqref{right}, we get that $(y^{(j)})_{j\geq 1}$ is a Cauchy sequence  in  $(\H_2^c(\M), \H_\infty^c(\M))_{\theta,\g}$. On the other hand, \eqref{left} and the second item imply that  the limit of the Cauchy sequence $(y^{(j)})_{j\geq 1}$  in $(\H_2^c(\M), \H_\infty^c(\M))_{\theta,\g}$ must be $y$. 
 Thus,  by taking limits, we may conclude  that:
\begin{align*}
\big\|y\big\|_{(\H_2^c(\M), \H_\infty^c(\M))_{\theta,\g}} &=\lim_{j\to \infty} \big\|y^{(j)}\big\|_{(\H_2^c(\M), \H_\infty^c(\M))_{\theta,\g}}\\
&\lesssim_{q,\g} \lim_{j\to \infty} \big\|y^{(j)}\big\|_{\H_{q,\g}^c}\\
&=\big\|y\big\|_{\H_{q,\g}^c}.
\end{align*}
This shows that \eqref{right} is valid  for any $x \in \H_{q,\g}^c(\M)$ and combining with \eqref{left}, we conclude that
\[
\big(\H_2^c(\M), \H_\infty^c(\M)\big)_{\theta,\g} = \H_{q,\g}^c(\M).
\]

By reiteration, we may also state the slightly more general conclusion  that  if $2<r<\infty$, $0<\upsilon,\g\leq \infty$, $0<\theta<1$, and $1/q=(1-\theta)/r$, then 
\begin{equation*}
\big(\H_{r,\upsilon}^c(\M), \H_\infty^c(\M)\big)_{\theta,\g} = \H_{q,\g}^c(\M).
\end{equation*}

$\bullet$ {\it Step~2.} It is already known  that the result holds if both endpoints  consist  of Hardy spaces with finite indices (\cite{Musat-inter}). Therefore, it suffices to apply Wolff interpolation theorem. More specifically, we use  Lemma~\ref{union}.
For $1\leq \g\leq \infty$, 
set $A_{p,\g} := \H_{p,\g}^c(\M)$ when  $1\leq p<\infty$  and $A_{\infty, \g}:=\H_\infty^c(\M)$. 

By Step~1, 
 the family $\{ A_{p,\g}\}_{p,\g \in [1,\infty]}$ forms a real-interpolation scale on the interval $I=(2,\infty]$.
 On the other hand, from the  finite indices, the family $\{ A_{p,\g}\}_{p,\g \in [1,\infty]}$ forms a real-interpolation scale on the interval $J=[1,\infty)$.  Clearly, $|I\cap J|>1$. By Lemma~\ref{union}, we conclude that  the family $\{ A_{p,\g}\}_{p,\g \in [1,\infty]}$ forms a real-interpolation scale on  $I \cup J=[1,\infty]$.
This completes the proof.
\end{proof}

\begin{remark}\label{rem-general} 
The  argument used in  Step~1 of the proof above  can be easily adapted to provide the following more general statement:
if $E=(L_2, L_\infty)_{\cal{F};K}$ where $\cal{F}$ is a Banach function space with monotone norm and  the operator $C^{\ast}$ is bounded on $E$,  then 
\begin{equation}\label{E2}
\H_E^c(\M)=(\H_2^c(\M), \H_\infty^c(\M))_{\cal{F};K}
\end{equation}
with equivalent norms.
\end{remark}

\subsubsection{Extensions to Hardy spaces associated with general symmetric spaces}

We have the following result for Hardy spaces associated with symmetric spaces.
\begin{theorem}\label{lifting}
Assume that   $E \in {\rm Int}(L_2, L_\infty)$ and $C^*: E \to E$ is bounded. Let
$\cal{G}$ be a Banach function space with monotone norm.
 If $F= (E,L_\infty)_{\cal{G};K}$ and $F$ is $r$-concave for some $r<\infty$,  then 
\[
\H_F^c(\M)= (\H_E^c(\M),\H_\infty^c(\M))_{\cal{G};K}
\]
with equivalent norms.
\end{theorem}
The proof is based on the following  general reiteration for $K$-functionals.

\begin{proposition}[\cite{Ahmed-Kara-Reza}]\label{gen-Holm}
Let $(A_1,A_0)$ be a couple of quasi-Banach spaces  and $\cal{F}$  be a quasi-Banach function space with monotone quasi-norm. If $X=  (A_0, A_1)_{\cal{F}; K}$  and  $a \in X +A_1$,  then for every  $t>0$,
\[
K(a, \rho(t) ; X, A_1)  \approx I(t,a) +\frac{\rho(t)}{t} K(a, t; A_0, A_1),
\]
where $I(t,a)=\|\ch_{(0,t)}(\cdot) K(a, \cdot \,; A_0, A_1)\|_{\cal F}$ and 
$\rho(t)\approx t\|\ch_{(t,\infty)}(\cdot)\|_{\cal{F}} + \| u\mapsto u\chi_{(0,t)}(u)\|_{\cal{F}}$.
\end{proposition}

\begin{proof}[Proof of Theorem~\ref{lifting}]
We begin by verifying that  Theorem~\ref{main} extends to the  present situation. 
That is, if $x\in \H_E^c(\M) +\H_\infty^c(\M)$ is such that $S_c(x) \in \Lambda_{\log}(\M)$, then for every $s>0$, the following holds:
\begin{equation}\label{gen-K}
K\big(x, s; \H_E^c(\M),\H_\infty^c(\M)\big) \leq C_{abs} K\big( \mu(S_c(x)) +C^{\ast}[\mu(S_c(x))],s; E, L_\infty\big). 
\end{equation}
By assumption, we have from Remark~\ref{rem-general} that $\H_E^c(\M)=\big(\H_2^c(\M), \H_\infty^c(\M)\big)_{\cal{F}; K}$.  Therefore by Proposition~\ref{gen-Holm}, we have  for $t>0$, 
\[
K(x, \rho(t) ; \H_E^c(\M), \H_\infty^c(\M))  \approx I(t,x) +\frac{\rho(t)}{t} K(x, t; \H_2^c(\M), \H_\infty^c(\M))
\]
where  $I(t,x)=\|\ch_{(0,t)}(\cdot) K(x, \cdot \,; \H_2^c(\M), \H_\infty^c(\M))\|_{\cal F}$. It then follows from Theorem~\ref{main} and the monotonicity of $\|\cdot\|_{\cal{F}}$ that:
\[
I(t,x)  \lesssim \|\ch_{(0,t)}(\cdot) K(\mu(S_c(x))+C^{\ast}[\mu(S_c(x))], \cdot \,; L_2, L_\infty)\|_{\cal F}:=\widehat{I}(t, \mu(S_c(x)) +C^{\ast}[\mu(S_c(x))]).
\] 
We can then deduce that 
\begin{align*}
K(x, \rho(t) &; \H_E^c(\M), \H_\infty^c(\M)) \\
&\lesssim \widehat{I}(t,  \mu(S_c(x))+C^{\ast}[\mu(S_c(x))]) +\frac{\rho(t)}{t} K( \mu(S_c(x))+C^{\ast}[\mu(S_c(x))], t; L_2, L_\infty)\\
&\approx K(\mu(S_c(x))+C^{\ast}[\mu(S_c(x))], \rho(t); E, L_\infty).
\end{align*}
One can  verify as in the proof of \cite[Theorem~3.13]{Ran-Int2} that  the range of $\rho(\cdot)$ is $[0,\infty)$ which  proves \eqref{gen-K}. 

In turn,  the estimate \eqref{gen-K}  implies that if $x \in \H_E^c(\M)  + \H_\infty^c(\M)$ is such that $S_c(x)\in \Lambda_{\log}(\M)$, then 
\begin{align*}
\big\| x\big\|_{(\H_E^c(\M),\H_\infty^c(\M))_{\cal{G}; K}} &\lesssim \big\|\mu(S_c(x)) + C^{\ast}[\mu(S_c(x))]\big\|_{(E,L_\infty)_{\cal{G}; K}}\\
&\approx \big\| \mu(S_c(x))+C^{\ast}[ \mu(S_c(x)]\big\|_F.
\end{align*}
Next, since $F$ is $r$-concave, we have $F\in {\rm Int}(L_2, L_r)$. It follows from Lemma~\ref{C-bounded} and interpolation that $C^{\ast}$ is a  bounded   operator on $F$.  This implies further that 
\[
\big\| x\big\|_{(\H_E^c(\M),\H_\infty^c(\M))_{\cal{G}; K}}\leq C_E \big\| Id+C^{\ast}: F \to F\big\|. \big\|x \big\|_{\H_F^c}.
\]
As above, the extra assumption that $x\in \H_{\Lambda_{\log}}^c(\M)$ can  be removed by approximations  for which we omit the details.
\end{proof}

\smallskip

Our argument  above  is  clearly handicapped  by  the fact that we only have estimate  on  the $K$-functional  for the couple $(\H_2^c(\M),\H_\infty^c(\M))$.   We suspect that this extra assumption is not necessary.  We leave  as an open problem that  Theorem~\ref{lifting} can be improved  to    cover all  spaces  $E\in {\rm Int}(L_1,L_\infty)$.

\smallskip

As an illustration, we treat the case of  martingale Orlicz-Hardy spaces. In this special situation, the restriction in Theorem~\ref{lifting} is not needed.
We start from recalling that   at the level of function spaces,  the following result holds:
\begin{proposition}[{\cite[Proposition~3.3]{L-T-Zhou}}]\label{function}
Let $\Phi$ be an Orlicz function, $1\leq\g\leq \infty$, and  $0<\theta<1$. If  $\Phi_0^{-1}(t) =[\Phi^{-1}(t)]^{1-\theta}$, then 
\[
(L_\Phi, L_\infty)_{\theta, \g} =L_{\Phi_0,\g}.
\]
\end{proposition}

By reiteration,    we also deduce  the following:   assume that   $0< \theta, \eta<1$ and $1\leq \lambda,  \gamma \leq \infty$. Set   $\Psi_1$ and $\Psi_2$ 
such that $\Psi_1^{-1}(t) =[\Phi^{-1}(t)]^{1-\theta}$ and $\Psi_2^{-1}(t) =[\Phi^{-1}(t)]^{1-\theta\eta}$.  Then
\[
\big( L_{\Phi}, L_{\Psi_1, \lambda}\big)_{\eta, \g} =L_{\Psi_2,\g}.
\]
Next, we recall  that  $\H_p^c(\M)$ is complemented in $L_p^{\rm ad}(\M;\ell_2^c)$ for $1\leq p<\infty$. Since the above identity transfers to the corresponding spaces of adapted sequences (see \cite{Ran-Int2}), 
 it follows that  if $\Phi$ and $\Psi_1$ are convex  Orlicz functions that are also  $q$-concave for some $1\leq q<\infty$, then  the following holds:
\begin{equation}\label{iteration}
\big( \H_{\Phi}^c(\M), \H_{\Psi_1, \lambda}^c(\M)\big)_{\eta, \g} =\H_{\Psi_2,\g}^c(\M).
\end{equation}
We may view this as the Orlicz extension of having finite indices.

\smallskip

The  aim of the next  result  is to show that  as in conditioned case,  the full    equivalence  in  Proposition~\ref{function} transfers to  martingale Hardy spaces. This provides an extension of Theoorem~\ref{main-interpolation} to the case of martingale  Orlicz-Hardy spaces.

 \begin{theorem}\label{Orlicz}
 Let $0<\theta <1$ and  $1\leq  \g \leq \infty$. If $\Phi$ is a convex  Orlicz function that is  $q$-concave for $1\leq q<\infty$,  then   for  $\Phi_0^{-1}(t) =[\Phi^{-1}(t)]^{1-\theta}$, the following  holds:
 \[
\big( \H_{\Phi}^c(\M), \H_\infty^c(\M)\big)_{\theta, \g} =\H_{\Phi_0,\g}^c(\M).
\]
\end{theorem}

\begin{proof}  We will consider three cases.

$\bullet$ Case~1. Assume first that $\Phi$ is $p$-convex and $q$-concave for $2<p\leq q<\infty$.  In this case, $L_\Phi \in {\rm Int}(L_2, L_q)$ and  the statement follows immediately from Theorem~\ref{lifting} and Proposition~\ref{function}. In fact, we have the following  slightly more general statement:
\[
\big( \H_{\Phi,\lambda}^c(\M), \H_\infty^c(\M)\big)_{\theta, \g} =\H_{\Phi_0,\g}^c(\M).
\]

\smallskip

$\bullet$  Case~2.  Assume that  $1/2<\theta <1$.  One can easily see that $\Phi_0$ is $(1-\theta)^{-1}$-convex and $q(1-\theta)^{-1}$-concave. Moreover, $(1-\theta)^{-1}>2$.
Fix $1/2<\psi<\theta$ and define $\Phi_1$ so  that 
\[
\Phi_1^{-1}(t) =[\Phi^{-1}(t)]^{1-\psi}, \quad t>0.
\]
We note that $\Phi_1$ is $(1-\psi)^{-1}$-convex  with $(1-\psi)^{-1}>2$ and  $q(1-\psi)^{-1}$-concave. Moreover,  $\Phi_0^{-1}(t)=[ \Phi_1^{-1}(t)]^{1-\theta_0}$ for $\displaystyle{1-\theta_0=\frac{1-\theta}{1-\psi}}$.  Applying Case~1 to  $\Phi_1$, we have 
\begin{equation*}
\big(\H_{\Phi_1,\lambda}^c(\M), \H_\infty^c(\M)\big)_{\theta_0, \gamma} =\H_{\Phi_0,\gamma}^c(\M).
\end{equation*}
On the other hand, we also have from \eqref{iteration} that
\begin{equation*}
\big(\H_{\Phi}^c(\M), \H_{\Phi_0, \gamma}^c(\M)\big)_{\theta_1, \lambda} =\H_{\Phi_1,\lambda}^c(\M)
\end{equation*}
where $\theta_1= \psi/\theta$.  By Wolff's interpolation theorem stated in Theorem~\ref{W} with $B_1=\H_{\Phi}^c(\M)$, $B_2=\H_{\Phi_1,\lambda}^c(\M)$,
$B_3=\H_{\Phi_0,\g}^c(\M)$, and $B_4=\H_{\infty}^c(\M)$,  we concude that  
\begin{equation*}
\big(\H_{\Phi}^c(\M), \H_\infty^c(\M)\big)_{\xi, \lambda} =\H_{\Phi_0,\lambda}^c(\M)
\end{equation*}
where $\displaystyle{\xi=\frac{\theta_0}{1-\theta_1 +\theta_1\theta_0}}$. A simple calculation shows that $\xi=\theta$.

\smallskip

$\bullet$ Case~3.  Assume that $0<\theta\leq 1/2$. Fix $\psi$ so that $0<1-\psi<\theta \leq 1/2$. 
Set $\widetilde{\Phi}$ satisfying   $\widetilde{\Phi}^{-1}(t) = [\Phi_0^{-1}(t)]^{1-\psi}$ for $t>0$. Note that  $1/2<\psi<1$.   Using Case~2 with $\Phi_0$ in place of $\Phi$ and $\widetilde{\Phi}$ in place of $\Phi_0$, we get
\begin{equation*}
\big(\H_{\Phi_0,\lambda}^c(\M), \H_\infty^c(\M)\big)_{\psi, \lambda} =\H_{\widetilde{\Phi},\lambda}^c(\M).
\end{equation*}
Next, since  for  every $t>0$,  $\widetilde{\Phi}(t)^{-1}= [ \Phi^{-1}(t)]^{(1-\theta)(1-\psi)}=[ \Phi^{-1}(t)]^{(1-\tilde{\theta})}$, it follows  \eqref{iteration} that
\begin{equation*}
\big(\H_{\Phi}^c(\M), \H_{\widetilde{\Phi},\lambda}^c(\M)\big)_{\eta, \lambda} =\H_{\Phi_0,\lambda}^c(\M)
\end{equation*}
where $\displaystyle{\eta= \frac{\theta}{\tilde{\theta}}=\frac{\theta}{\theta+\psi-\psi\theta}}$. We use  Wolff's interpolation theorem with $B_1=\H_{\Phi}^c(\M)$, $B_2=\H_{\Phi_0,\lambda}^c(\M)$,  $B_3=\H_{\widetilde{\Phi},\lambda}^c(\M)$, and $B_4=\H_{\infty}^c(\M)$ to conclude that 
\begin{equation*}
\big(\H_{\Phi}^c(\M), \H_\infty^c(\M)\big)_{\upsilon, \lambda} =\H_{\Phi_0,\lambda}^c(\M)
\end{equation*}
where 
$\displaystyle{\upsilon= \frac{\eta\psi}{1-\eta+\eta\psi}}$. One can easily verify  that $\upsilon=\theta$. The proof is complete.
\end{proof}

We conclude   this section with the  corresponding result for  BMO-spaces. This may be viewed as an Orlicz generalization of  the real interpolation form of Musat's result (\cite{Musat-inter}). To the best of our knowledge, the only available result in the literature is for classical martingale Hardy spaces associated with regular filtration  
(see \cite[Corollary~4.9]{L-Weisz-Xie}).  The proof outlined below is based on  interpolation of spaces of adapted sequences.  

 \begin{theorem}\label{Orlicz-2}
 Let $0<\theta <1$ and  $1\leq  \g \leq \infty$. If $\Phi$ is a convex  Orlicz function that is  $q$-concave for $1\leq q<\infty$,  then   for  $\Phi_0^{-1}(t) =[\Phi^{-1}(t)]^{1-\theta}$, the following  holds:
 \[
\big( \H_{\Phi}^c(\M), \BMO^c(\M)\big)_{\theta, \g} =\H_{\Phi_0,\g}^c(\M).
\]
\end{theorem}

\begin{proof}[Sketch  of  the proof]
Assume first that $\Phi$ is $p$-convex and $q$-concave for $1<p<q<\infty$. Denote by $\Phi^*$  (resp. $\Phi_0^*$) the Orlicz function complementary to  the convex function $\Phi$ (resp. $\Phi_0$). Then $\Phi^*$ is $q'$-convex and $p'$-concave where $p'$ and $q'$ denote the conjugate indices of $p$ and $q$ respectively. 
In this case, $L_{\Phi^*} \in {\rm Int}(L_{q'}, L_{p'})$. A fortiori,  $L_{\Phi^*} \in {\rm Int}(L_{p'}, L_{1})$. Let $\cal{F}$ be a Banach function space with monotone norm so that $L_{\Phi^*}=(L_{p'}, L_1)_{\cal{F}; K}$.  The existence of such $\cal{F}$ is given   by Proposition~\ref{K-method}.  It follows that  $L_{\Phi^*}(\M;\ell_2^c)=(L_{p'}(\M;\ell_2^c), L_1(\M;\ell_2^c))_{\cal{F}; K}$. Similarly, $L_{\Phi^*}^{\rm ad}(\M;\ell_2^c)=(L_{p'}^{\rm ad}(\M;\ell_2^c), L_1^{\rm ad}(\M;\ell_2^c))_{\cal{F}; K}$.

By Proposition~\ref{gen-Holm}, one can express the $K$-functionals of the couple $(L_{\Phi^*}(\M;\ell_2^c), L_1(\M;\ell_2^c))$  (resp.  $L_{\Phi^*}^{\rm ad}(\M;\ell_2^c), L_1^{\rm ad}(\M;\ell_2^c))$)  in terms of those  of the couple  $(L_{p'}(\M;\ell_2^c), L_1(\M;\ell_2^c))$ (resp. $(L_{p'}^{\rm ad}(\M;\ell_2^c), L_1^{\rm ad}(\M;\ell_2^c))$).

The important fact here is that  the couple $(L_{p'}^{\rm ad}(\M;\ell_2^c), L_1^{\rm ad}(\M;\ell_2^c))$ is  $K$-closed in the  couple  $(L_{p'}(\M;\ell_2^c), L_1(\M;\ell_2^c))$ (see \cite[Proposition~3.19]{Ran-Int2}). 
Therefore, we also have that   the couple $(L_{\Phi^*}^{\rm ad}(\M;\ell_2^c), L_1^{\rm ad}(\M;\ell_2^c))$ is  $K$-closed in the  larger couple  $(L_{\Phi^*}(\M;\ell_2^c), L_1(\M;\ell_2^c))$.

From $K$-closedness, we deduce that if $(L_{\Phi^*}, L_1)_{\theta, \g'}=E$,  then 
\[
(L_{\Phi^*}^{\rm ad}(\M;\ell_2^c), L_1^{\rm ad}(\M;\ell_2^c))_{\theta,\g'}=E^{\rm ad}(\M;\ell_2^c).
\]
By complementation, we obtain further that 
\[
(\H_{\Phi^*}^c(\M), \H_1^c(\M))_{\theta, \g'}= \H_E^c(\M).
\]
With the facts that  $\H_{\Phi^*}^c(\M)= (\H_{\Phi}^c(\M))^*$,  $\H_{E^*}^c(\M)=(\H_{E}^c(\M))^*$, and $(\H_{1}^c(\M) )^*=\BMO^c(\M)$, we may apply duality   (see \cite[Theorem~3.7.1]{BL}) to deduce  that 
\[
(\H_{\Phi}^c(\M), \BMO^c(\M))_{\theta, \g}= \H_{E^*}^c(\M).
\]
Observe that $E^*=(L_{\Phi}, L_\infty)_{\theta, \g}= L_{\Phi_0,\g}$.
 This proves the theorem for  the case $\Phi$ being $p$-convex  with $p>1$.
 
  For general convex function $\Phi$ as stated, one can repeat  the argument used in  Case~2 of the proof of Theorem~\ref{Orlicz}  (but without the restriction $1/2<\theta<1$).  We leave the details to the  interested reader.
\end{proof}

\section{Concluding remarks}
By taking adjoints, all results from  the previous section are valid for  noncommutative martingale row Hardy spaces.

Recall that  the mixed martingale Hardy spaces are defined as follows: 
for $E\in {\rm Int}(L_p, L_2)$ 
with  $0<p<2$, 
\[
\H_E(\M)=\H_E^c(\M) + \H_E^r(\M)
\]
while for $F\in {\rm Int}(L_2, L_q)$ 
with  $2<q \leq \infty$, 
\[
\H_F(\M)=\H_F^c(\M) \cap \H_F^r(\M).\]

Using similar argument  as in the  proof of \cite[Theorem~4.5]{Bekjan-Chen-Perrin-Y}, we may also deduce  the corresponding  interpolation result for  mixed Hardy spaces which reads as follows:

\begin{theorem} If $0<\theta <1$, $1/p=1-\theta$, and  $1\leq \g \leq \infty$, then
\[
\big( \H_{1}(\M), \H_\infty(\M)\big)_{\theta, \g} = \H_{p,\g}(\M)
\]
with equivalent norms.
\end{theorem}

Motivated by  results from the previous section and the Musat's  result on the complex interpolation of $(\H_1^c(\M), \BMO^c(\M))$,  a natural  direction of interest   is the complex interpolation method for the couple $(\H_1^c(\M),\H_\infty^c(\M))$ (or  the couple $(\H_1(\M),\H_\infty(\M))$). For a given compatible couple of Banach spaces  $(A_0,A_1)$ and $0<\theta<1$, let $[A_0,A_1]_\theta$ denote the complex interpolation space of exponent $\theta$ as defined in \cite{BL}. The following question  remains unresolved.
\begin{problem}
Assume that $1<\theta<1$ and $1/p=1-\theta$. 
Does one have $[\H_1^c(\M), \H_\infty^c(\M)]_\theta=\H_p^c(\M)$?
\end{problem}
We should point out  that the corresponding problem for the conditioned case is still open (see \cite[Problem~5]{Bekjan-Chen-Perrin-Y}).
\bigskip 

 \noindent{\bf Acknowledgments.} I am very grateful to  the  anonymous referee  for a careful reading of the paper and  for  providing valuable suggestions that  improved the presentation of the paper.


\def\cprime{$'$}
\providecommand{\bysame}{\leavevmode\hbox to3em{\hrulefill}\thinspace}
\providecommand{\MR}{\relax\ifhmode\unskip\space\fi MR }
\providecommand{\MRhref}[2]{%
  \href{http://www.ams.org/mathscinet-getitem?mr=#1}{#2}
}
\providecommand{\href}[2]{#2}

\end{document}